\documentclass{amsart}
\usepackage{amsfonts,amssymb,amsmath,amsthm}
\usepackage{url}
\usepackage{enumerate}
\usepackage{verbatim}
\usepackage{mathrsfs}

\newtheorem{thm}{Theorem}[section]
\newtheorem{mainthm}{Main Theorem}
\newtheorem{prop}[thm]{Proposition}
\newtheorem{cor}[thm]{Corollary}
\newtheorem{lem}[thm]{Lemma}

\theoremstyle{definition}
\newtheorem{defn}[thm]{Definition}

\newtheorem{qu}{Question}
\newtheorem{notn}[thm]{Notation}
\newtheorem{rem}[thm]{Remark}
\newtheorem{obs}[thm]{Observation}
\numberwithin{equation}{section}

\begin{document}
\title{The Topological Complexity of a Surface}
\author{Aldo-Hilario Cruz-Cota}
\address{Department of Mathematics and Statistics, University of Nebraska at Kearney, Kearney, NE 68845, USA}
\email{cruzcotaah@unk.edu}

\begin{abstract} 

\noindent Let $p$ be a branched covering of a Riemann surface to the Riemann sphere $\mathbb{P}^1$, with branching set $B \subset \mathbb{P}^1$. We define the \emph{complexity} of $p$ as infinity, if $\mathbb{P}^1 \setminus B$ does not admit a hyperbolic structure, or the product of its degree and the hyperbolic area of $\mathbb{P}^1 \setminus B$, otherwise. The  \emph{topological complexity} of a surface $S$ is defined as the infimum of the set of all complexities of branched coverings $M \to \mathbb{P}^1$, where $M$ is a Riemann surface homeomorphic to $S$. We prove that if $S$ is a connected, closed, orientable surface of genus $g$, then its topological complexity, $C_{\text{top}}(S)$, is given by:

\[C_{\text{top}}(S)=
\left\{
	\begin{array}{cl}
		2\pi(2g+1)   & \mbox{if } g \geq 1,\\
		 6 \pi  & \mbox{if } g=0.	
	\end{array}
\right.\]
\end{abstract}
\maketitle

\section{Introduction}

Given a surface $S$, in this article we study all the possible ways in which a Riemann surface homeomorphic to $S$  can be realized as the branched covering of the Riemann sphere $\mathbb{P}^1$. To that end, we use a generalization of the complexity function for branched coverings introduced in \cite{CC-RR-BAMS}, which is defined as follows. Given a branched covering of the Riemann sphere $\mathbb{P}^1$, let $X$ denote the complement of the branching set in $\mathbb{P}^1$. Since $X$ is a domain of the Riemann sphere, it naturally has the structure of a Riemann surface. If $X$ admits a hyperbolic structure, we define the \emph{complexity} of the branched covering as the product of its degree and the hyperbolic area of $X$. If $X$ does not admit a hyperbolic structure, we define the \emph{complexity} of the branched covering to be infinity.

For a (connected, closed, orientable) Riemann surface $M$, we define its \emph{complexity} as the infimum of the set of all the complexities of  branched coverings $M \to \mathbb{P}^1$. We extend the definition of complexity to topological surfaces as follows. If $S$ is a (connected, closed, orientable) surface, we define its \emph{topological complexity} as the infimum of the set of all the complexities of  the Riemann surfaces that are homeomorphic to $S$. 

In \cite{CC-RR-BAMS}, the authors found a formula for the complexity of a Riemann surface, but that formula is in terms of an integer that is hard to find in practice, mainly because of its close connection to the \emph{Hurwitz problem} for the sphere. This problem is very difficult: it has been open for more than a century and is still open today in its full generality (see the next paragraph). The raison d'\^{e}tre for this article was to find  explicit formulas for the topological complexity of a surface that did not involve using the general Hurwitz problem. One of the main ingredients in the proofs of these formulas was adapting a variety of partial solutions to the Hurwitz problem to our context.

The Hurwitz problem for the sphere essentially asks when a collection of partitions of a positive integer can be realized as the combinatorial data of a branched covering of the sphere (see Section \ref{sect-Prelim} for precise statements). This problem was first studied by Hurwitz in \cite{Hurwitz}, although in the more general case in which the sphere is replaced by an arbitrary surface. Since then, many people have studied the Hurwitz problem, such as the authors of \cite{Edmonds}, \cite{Gersten},  \cite{Husemoller}, \cite{Ezell}, \cite{Baranski}, \cite{Pervova-Petronio-I}, \cite{Pervova-Petronio-II} and \cite{Pakovich}. In particular, we now know that the problem has been solved for almost all surfaces. The only instances of the problem that remain open reduce to the Hurwitz problem for the sphere (see Section 2 of \cite{Pervova-Petronio-I}). 

We now state the main results of this article.

\begin{mainthm} \label{main-thm-intro-top-comp}
Let $S$ be a connected, closed, orientable surface of genus $g$. Let $C_{\text{top}}(S)$ denote the topological complexity of $S$. Then 

\[C_{\text{top}}(S)=
\left\{
	\begin{array}{cl}
		2\pi(2g+1)   & \mbox{if } g \geq 1,\\
		 6 \pi  & \mbox{if } g=0.	
	\end{array}
\right.\]
\end{mainthm}

Following \cite{CC-RR-BAMS}, we will say that a branched covering $p \colon M \to \mathbb{P}^1$ of degree $d$ is \emph{simple} if the cardinality of $p^{-1}(y)$ is at least $d-1$ for all $y \in \mathbb{P}^1$. We define the \emph{simple topological complexity} of a (connected, closed, orientable) surface $S$ as the infimum of the complexities of all simple branched coverings $M \to \mathbb{P}^1$ with $M$ homeomorphic to $S$. Our second main result gives us explicit formulas for the simple topological complexity of a surface. 

\begin{mainthm} \label{main-thm-intro-simp-comp}
Let $S$ be a connected, closed, orientable surface of genus $g$. Let $C_{\text{simp}}(S)$ denote the simple topological complexity of $S$. Then 

\[C_{\text{simp}}(S)=
\left\{
	\begin{array}{cl}
		8\pi g   & \mbox{if } g \geq 1,\\
		 12\pi  & \mbox{if } g=0.	
	\end{array}
\right.\]
\end{mainthm}

Our proofs of the two main theorems are constructive; they exhibit the combinatorial data of branched coverings of minimal complexity (among those coverings whose domains have a fixed topological type). 

\begin{notn} In this paper we will only consider Riemann surfaces that are connected, closed and orientable.  Such surface  will be denoted by $M$, except in the case of the Riemann sphere, which we denote by $\mathbb{P}^1$. A topological surface will be denoted by $S$; the topological $2$-sphere will be denoted by $\mathbb{S}^2$. We use $\chi(M)$ to denote the Euler characteristic of the surface $M$. The infimum of a subset $A$ of the extended real line $[-\infty,\infty]$ will be denoted by $\inf(A)$.  The complex plane is denoted by $\mathbb{C}$. A sequence of positive integers in square brackets represents the elements of a partition. For example, $[1,1,1]$ denotes the trivial partition $3=1+1+1$. The cardinality of a set $B$ will be denoted by $|B|$.
\end{notn}

\section{Preliminaries} \label{sect-Prelim}

Most of the material in this section is contained in \cite{CC-RR-BAMS}. We include it here for the reader's  convenience.

\begin{defn} \label{Riemann-surf} 
A \emph{Riemann surface} is a complex manifold of complex dimension one.
\end{defn}

We now define the maps between Riemann surfaces that we will study in this article.

\begin{defn} \label{defn-branch-covering}
A \emph{branched covering} is a non-constant holomorphic map between connected Riemann surfaces.
\end{defn}

A branched covering is locally modeled by power functions.

\begin{prop} (\cite{Donaldson}) \label{branch-covering-as-power-fns}
 Let $p \colon M \to N$ be a branched covering between connected Riemann surfaces. Then for each $x \in M$ there is a unique integer $k=k_x \geq 1$ such that we can find charts around $x$ in $M$ (with $x$ corresponding to $0 \in \mathbb{C}$) and $p(x)$ in $N$ (with $p(x)$ corresponding to $0 \in \mathbb{C}$) in which the map $p$ is represented by the map $z \mapsto z^{k}$.
\end{prop}

The integer $k_x \geq 1$ is called the \emph{ramification index} of $x \in M$. A point in $M$ with ramification index greater than $1$ is called a \emph{ramification point}. The \emph{ramification set} is the set of all ramification points. This set is a discrete subset of $M$; it is a finite set when $M$ is compact. The \emph{branching set} is the image of the ramification set under the branched covering. A point in the branching set  is called a \emph{branch point}.

Given a branched covering $p \colon M \to N$, we can find an integer $d \geq 1$ such that every point in $N$ has exactly $d$ pre-images in $M$ (if we count them with the multiplicities given by their ramification indices).

\begin{lem} (\cite{Donaldson}) \label{lem-degree}
\label{lem-counting mult} Let $p \colon M \to N$ be a branched covering between connected, compact Riemann surfaces. Then there exists a positive integer $d$ such that \[\displaystyle \sum_{x\in p^{-1}(y)} k_x=d\] for every $y \in N$. The integer $d$ is called the \emph{degree} of the branched covering.
\end{lem}

For the rest of this article, we will only study branched coverings of the Riemann sphere $\mathbb{P}^1$. The following notation will be used frequently.

\begin{defn}
A \emph{$(d,n)$-branched covering}  is a  branched covering $p \colon M \to \mathbb{P}^1$ that satisfies the following conditions:
\begin{enumerate}
 \item $M$ is a connected, closed, orientable Riemann surface.
 \item The degree of $p$ is $d$.
 \item The cardinality of the branching set of $p$ is $n$.
\end{enumerate}
\end{defn}

We will use the following version of the classical Riemann-Hurwitz formula for branched coverings of the Riemann sphere.

\begin{thm} \label{RH-form-alt}  Let $p \colon M \to \mathbb{P}^1$ be a $(d,n)$-branched covering with branching set $B$. If the pre-image $p^{-1}(B)$ has cardinality $m$, then \[\chi(M)-m=d(\chi(\mathbb{P}^1)-n).\]
\end{thm}

Let $p \colon M \to \mathbb{P}^1$ be a $(d,n)$-branched covering. Suppose that $y_1,y_2,\cdots y_n$ are all the distinct branch points of the branched covering. For each $i=1,2,\cdots, n$, Lemma \ref{lem-degree} implies that the ramification indices of the points in $p^{-1}(y_i)$ form a partition $\Pi_i$ of the integer $d$. We gather all these partitions in a single collection $\Pi=\{\Pi_1,\Pi_2,\cdots,\Pi_n\}$ (with repetitions allowed). The triplet $(d,n,\Pi)$ contains all the combinatorial information about the branch points of the covering.

\begin{defn} Given a $(d,n)$-branched covering $p \colon M \to \mathbb{P}^1$, the triplet  $(d,n,\Pi)$ obtained as explained above is called the \emph{branch triplet} of the branched covering.
\end{defn}

We also consider collections of partitions that are not necessarily associated to branched coverings.

\begin{defn}
The triplet $(d,n,\Pi)$ is called an \emph{abstract branch triplet} if the following conditions are satisfied:
\begin{itemize}
 \item $d$ and $n$ are positive integers;
 \item $\Pi$ is a collection (with repetitions allowed) of $n$ partitions of $d$, each of which is not the trivial partition $[1,1,\cdots,1]$ of $d$.
\end{itemize}
The sum of the lengths of the partitions in $\Pi$ is called the \emph{total length} of the abstract branch triplet.
\end{defn}

We are primarily interested in abstract branch triplets that arise as branch triplets of branched coverings of the Riemann sphere $\mathbb{P}^1$.

\begin{defn} \label{defn-realizable} Let $\mathcal{T}$ be an abstract branch triplet and let $M$ be a connected, closed, orientable Riemann surface. We say that $\mathcal{T}$ is \emph{realizable} on $M$ if there exists a branched covering $p \colon M \to \mathbb{P}^1$  whose branch triplet is equal to $\mathcal{T}$. An abstract branch triplet is called  \emph{realizable} if it is realizable on a connected, closed, orientable Riemann surface.
\end{defn}

A natural question to ask is the following:

\begin{qu}
 What abstract branch triplets are realizable?
\end{qu}

The Riemann-Hurwitz formula gives a necessary condition for an abstract branch triplet to be realizable. For if  $(d,n,\Pi)$ is an abstract branch triplet of total length $m$ that is realizable on a (connected, closed, orientable) Riemann surface $M$, then the Riemann-Hurwitz formula (Theorem \ref{RH-form-alt}) implies that:
\begin{equation} \label{defn-comp-BT}
 \chi(M)-m=d(2-n).
\end{equation}

\begin{defn} \label{defn-comp}
An abstract branch triplet $(d,n,\Pi)$ of total length $m$ that satisfies equation \eqref{defn-comp-BT} is called \emph{compatible} with $M$. 
\end{defn}

Every abstract branch triplet that is realizable on a Riemann surface $M$ is  compatible with $M$. However, the converse of the last statement is not true (see \cite[Corollary 6.4]{Edmonds}). The problem of determining what compatible branch triplets are realizable as branched coverings of the sphere is known as the \emph{Hurwitz problem} for the sphere. This problem was first studied by Hurwitz in \cite{Hurwitz} and, more recently, many other mathematicians, such as the authors of \cite{Edmonds}, \cite{Gersten},  \cite{Husemoller}, \cite{Ezell}, \cite{Baranski}, \cite{Pervova-Petronio-I}, \cite{Pervova-Petronio-II} and \cite{Pakovich}. The problem is still open in its full generality (see Section 2 of \cite{Pervova-Petronio-I}). 

In the rest of the section we collect the partial solutions to the Hurwitz problem that we need to prove our main theorems. We start with a classical result of Edmonds, Kulkarni and Stong.

\begin{thm} \label{thm-EKS} (\cite[Proposition $5.2$]{Edmonds}) Let $\mathcal{T}=(d,n,\Pi)$ be an abstract branch triplet of total length $m$. Suppose that the length of one of the partitions in $\Pi$ is one, and that  $nd-m$ is an even number such that $nd-m \geq 2d-2$. Then $\mathcal{T}$ is realizable.
\end{thm}

The next two results are corollaries of Theorem \ref{thm-EKS}.

\begin{lem} \label{lem-ABT-realiz-sph-1}
The abstract branch triplet $\mathcal{T}=(3,3,\{[3],[1,2],[1,2]\})$ is realizable on a Riemann surface homeomorphic to the sphere.
\end{lem}

\begin{thm} \label{thm-BT-realizable}
 Let $S$ be a connected, closed, orientable surface of genus $g\geq 1$ and let $d=2g+1$. Then the abstract branch triplet $\mathcal{T}=(d,3,\{[d],[d],[d]\})$ is realizable on a Riemann surface that is homeomorphic to  $S$.
\end{thm}

In \cite[Theorem $12$]{Baranski}, Bara\'{n}ski proved the following (see also \cite[Proposition $2.11$]{Pervova-Petronio-I}).

\begin{thm} \label{thm-Bar} (\cite[Theorem $12$]{Baranski}) Let $\mathcal{T}=(d,n,\Pi)$ be an abstract branch triplet with $n \geq d$. Suppose that $\mathcal{T}$ is compatible with the topological $2$-sphere $\mathbb{S}^2$. Then $\mathcal{T}$ is realizable by a Riemann surface homeomorphic to $\mathbb{S}^2$.
\end{thm}

\section{Different Notions of Complexity} \label{sect-diff-not-comp}

We define several notions of complexity: one for branched coverings, another for Riemann surfaces and the last one for topological surfaces. The first two concepts were essentially introduced in \cite{CC-RR-BAMS}. 

\begin{defn} \label{defn-comp-BC} Let $p \colon M \to \mathbb{P}^1$ be a branched covering of degree $d$ of the  (connected, closed, orientable) Riemann surface $M$ to the Riemann sphere $\mathbb{P}^1$. Let $B \subset \mathbb{P}^1$ be the branching set of $p$. Then the complement $\mathbb{P}^1 \setminus B$ of $B$ in $\mathbb{P}^1$ inherits the structure of a Riemann surface from the Riemann sphere $\mathbb{P}^1$.
 
\begin{enumerate}
 \item We call the branched covering $p \colon M \to \mathbb{P}^1$ \emph{hyperbolic} if $\mathbb{P}^1 \setminus B$ admits a hyperbolic structure. In this case, we denote the hyperbolic area of $\mathbb{P}^1 \setminus B$ by $\mathcal{A}(\mathbb{P}^1 \setminus B)$.
 \item We define the \emph{complexity} $C_\text{Cov}(p)$ of the branched covering $p$ as follows:
\[C_\text{Cov}(p)=
\left\{
	\begin{array}{cl}
		d \cdot \mathcal{A}(\mathbb{P}^1 \setminus B)   & \mbox{if } p \text{ is hyperbolic},\\
		\infty  & \mbox{otherwise}.	
	\end{array}
\right.\]
\end{enumerate}

\end{defn}

\begin{rem} \label{rem-hyp-n-geq-3} The Riemann surface $\mathbb{P}^1 \setminus B$ admits a hyperbolic structure if and only if the cardinality of the set $B$ is greater than or equal to $3$ (see \cite[Theorem 27.12]{Forster}). In other words, a $(d,n)$-branched covering is hyperbolic if and only if $n \geq 3$. 
\end{rem}

Definition \ref{defn-comp-BC} allows us to extend the definition of complexity given in \cite{CC-RR-BAMS} to all $(d,n)$-branched coverings, including the ones that are not hyperbolic.

The complexity of a hyperbolic branched covering is a mixture of a topological invariant, the degree, and a geometric invariant, the hyperbolic area of the complement of the  branching set in $\mathbb{P}^1$. Thus, we can use the Gauss-Bonnet theorem to find an explicit formula for the complexity of a hyperbolic branched covering. [For a statement of the  Gauss-Bonnet theorem, see for example \cite[Theorem V.2.7]{Chavel}.] This is done in the following lemma.


\begin{lem} \label{lem-area-complex} \hfill
\begin{enumerate}
 \item The hyperbolic area of the complement of $n \geq 3$ points in the Riemann sphere equals $2\pi(n-2)$. 
 \item  The complexity of a $(d,n)$-branched covering with $n \geq 3$ is finite and equal to  $2\pi d (n-2)$.
\end{enumerate}
\end{lem}

\begin{proof} \hfill
\begin{enumerate}
 \item Let $\mathcal{A}(M_n)$ denote the hyperbolic area of the complement $M_n$ of $n \geq 3$ points in the Riemann sphere $\mathbb{P}^1$. It follows from the Gauss-Bonnet theorem that $\mathcal{A}(M_n)=-2\pi \chi(M_n)=-2\pi (\chi(\mathbb{P}^1)-n)=2\pi (n-2)$.
 \item Let $p \colon M \to \mathbb{P}^1$ be a $(d,n)$-branched covering with $n \geq 3$. Let $B$ be the branching set of the covering $p$. By Remark \ref{rem-hyp-n-geq-3}, $\mathbb{P}^1 \setminus B$ admits a hyperbolic structure. Thus, the complexity $C(p)$ of $p$ is defined as $C(p)=d \cdot \mathcal{A}(\mathbb{P}^1 \setminus B)$, where $\mathcal{A}(\mathbb{P}^1 \setminus B)$ denotes the hyperbolic area of $\mathbb{P}^1 \setminus B$. Using $(1)$, we obtain that $\mathcal{A}(\mathbb{P}^1 \setminus B)=2\pi(n-2)$, and so $C(p)=d \cdot \mathcal{A}(\mathbb{P}^1 \setminus B)=2\pi d (n-2)$.
\end{enumerate}
\end{proof}

Combining Definition \ref{defn-comp-BC}, Remark \ref{rem-hyp-n-geq-3} and Lemma \ref{lem-area-complex}, we obtain the following.

\begin{cor} \label{cor-hyp-BC-comp-fin} For a $(d,n)$-branched covering $p \colon M \to \mathbb{P}^1$, the following statements are equivalent:
\begin{enumerate}
 \item $p \colon M \to \mathbb{P}^1$ is hyperbolic;
 \item $n \geq 3$;
 \item $C_\text{Cov}(p)<\infty$.
\end{enumerate}
\end{cor}

Let $\mathscr{C}$ be the set of all complexities of branched coverings of connected, closed, orientable Riemann surfaces to the Riemann sphere. Let $\mathbb{Z} \pi$ be the set of all integer multiples of the number $\pi$. By Lemma \ref{lem-area-complex} and Corollary \ref{cor-hyp-BC-comp-fin}, $\mathscr{C}$ is a subset of $\mathbb{Z} \pi \cup \{\infty\}$. Since $\mathbb{Z} \pi$ is a discrete subset of the real line, then we have the following:

\begin{obs} \label{obs-non-empty-subset-C}
 The infimum of a non-empty subset of $\mathscr{C}$ is attained, although it could be infinite.
\end{obs}

We now define the complexity of a Riemann surface. It essentially measures the most economical way (in terms of complexity) in which the Riemann surface can be expressed as a branched covering of the Riemann sphere.

\begin{defn} \label{defn-comp-RS} Let $M$ be a connected, closed, orientable Riemann surface. We define the \emph{complexity} $C_{\text{Riem}}(M)$ of $M$ as the infimum of the set of all complexities of branched coverings of $M$ to the Riemann sphere: \[C_{\text{Riem}}(M)=\inf\{C_{\text{Cov}}(p)\,|\, p \colon M \to \mathbb{P}^1 \text{ is a branched covering}\}.\]
\end{defn}

It is a well-known fact that every compact Riemann surface is a branched covering of the Riemann sphere $\mathbb{P}^1$ (see for example \cite[Theorem 16.11]{Forster}). Hence, given a Riemann surface $M$, the set of all complexities of branched coverings of $M$ to $\mathbb{P}^1$  is not empty. Thus, by Observation \ref{obs-non-empty-subset-C}, the infimum in Definition \ref{defn-comp-RS} is  attained, although it could be infinite. In other words, there exists a branched covering $p \colon M \to \mathbb{P}^1$ such that $C_{\text{Riem}}(M)=C_{\text{Cov}}(p)$.

In  \cite{CC-RR-BAMS}, the authors find a formula for the complexity of a Riemann surface.

\begin{thm} ( \cite[Theorem 5.4]{CC-RR-BAMS}) \label{main-Thm-CC-RR} Let $M$ be a connected, closed, orientable Riemann surface of genus $g \geq 1$. Suppose that the complexity $C_{\text{Riem}}(M)$ of $M$ is finite. Let $m_{\text{min}}$ be the minimum total length of an abstract branch triplet that is realizable on $M$. Then \begin{equation} \label{eqn-stat-main-Thm-CC-RR}                                                                                                                                                                                                                                                                                                                         C_{\text{Riem}}(M)=2\pi(m_{\text{min}}+2g-2).                                                                                                                                                                                                                                                                                                                                  \end{equation}
\end{thm}

We will need to find a lower bound for the number $m_{min}$ from above. 

\begin{prop} \label{prop-mmin-bounded} The number $m_{\text{min}}$ from equation \ref{eqn-stat-main-Thm-CC-RR} is greater than or equal to $3$.  
\end{prop}

\begin{proof} Let $\mathcal{T}=(d,n,\Pi)$ be an abstract branch triplet that is realizable on $M$. We prove that the total length of $\mathcal{T}$ is always greater than or equal to three. Let
$p \colon M \to \mathbb{P}^1$ be the $(d,n)$-branched covering associated to $\mathcal{T}$.

Suppose that $y_1,y_2,\cdots y_n$ are all the distinct branch points of the branched covering $p$. For each $i=1,2,\cdots,n$,
\begin{enumerate}
 \item let $\Pi_i$ be the partition of $d$ given by the ramification indices of the points in $p^{-1}(y_i)$;
 \item let $m_i$ be the length of the partition $\Pi_i$.
\end{enumerate}

Since $m=m_1+m_2+\cdots+m_n$ and each $m_i \geq 1$ ($i=1,2,\cdots,n$), then $m \geq n$. It remains to prove that $n \geq 3$.

Since the abstract branch triplet $\mathcal{T}$ is compatible with $M$, then
\begin{equation} \label{eqn-prop-mmin-1}
m-\chi(M)=d(n-2).
\end{equation}

Also, $m>0$, $\chi(M)=2-2g \leq 0$ and $d>0$, so equation \eqref{eqn-prop-mmin-1} implies that  $n-2>0$, i.e., $n \geq 3$.
\end{proof}

The next definition introduces the simplest type of branched coverings that we will study in this article.

\begin{defn} \label{defn-hyperelliptic}
A compact Riemann surface $M$  is called \emph{hyperelliptic} if there exists a double branched covering $p \colon M \to \mathbb{P}^1$.
\end{defn}

Notice that our definition allows the existence of hyperelliptic Riemann surfaces of genus $0$ and $1$. 

Applying the Riemann-Hurwitz formula to a double branched covering we obtain the following.

\begin{lem} \label{lem-RH-hyp-surf}
 Let $p \colon M \to \mathbb{P}^1$ be a $(2,n)$-branched covering. If the genus of $M$ is equal to $g \geq 0$, then $n=2g+2$.
\end{lem}

\begin{notn}
Given a connected, closed, orientable surface $S$, we will use the symbol $\mathbb{X}_S$ to denote the set of all Riemann surfaces that are homeomorphic to $S$.
\end{notn}

The following is a well-known result (see for example \cite{Forster}).

\begin{prop} \label{prop-XS-hyper}
 Given a connected, closed, orientable surface $S$, there exists a hyperelliptic Riemann surface in $\mathbb{X}_S$.
\end{prop}

We now define a notion of complexity for topological surfaces.

\begin{defn} \label{defn-top-comp} Let $S$ be a connected, closed, orientable surface. We define the \emph{topological complexity}, $C_{\text{top}}(S)$, of $S$ as the infimum of the set of all complexities of the Riemann surfaces in $\mathbb{X}_S$: \[C_{\text{top}}(S)=\inf\{C_{\text{Riem}}(M)\,|\, M \in \mathbb{X}_S\}.\]
\end{defn}

Given a connected, closed, orientable surface $S$, Proposition \ref{prop-XS-hyper} shows that the set of all complexities of the Riemann surfaces in $\mathbb{X}_S$ is not empty. Thus, by Observation \ref{obs-non-empty-subset-C}, the infimum in Definition \ref{defn-top-comp} is  attained, although it could be infinite. In other words, there exists a Riemann surface $M \in \mathbb{X}_S$ such that $C_{\text{top}}(S)=C_{\text{Riem}}(M)$.

We devote the rest of this section to prove that the topological complexity of a surface is always finite. We start with the case of surfaces with positive genus.

\begin{prop} \label{prop-top-comp-fin} The topological complexity of a connected, closed, orientable surface of genus $g\geq 1$ is finite.
\end{prop}

\begin{proof}
Let $S$ be a connected, closed, orientable surface of genus $g\geq 1$. By Proposition \ref{prop-XS-hyper}, there exists a hyperelliptic Riemann surface $M \in \mathbb{X}_S$. Let $p \colon M \to \mathbb{P}^1$ be the associated $(2,n)$-branched covering. By Lemma \ref{lem-RH-hyp-surf}, $n=2g+2$. Since $g\geq 1$, then $n=2g+2 \geq 4$, and so $C_\text{Cov}(p)<\infty$ (by Corollary \ref{cor-hyp-BC-comp-fin}). Thus, \[C_{\text{top}}(S) \leq C_{\text{Riem}}(M) \leq C_\text{Cov}(p)<\infty.\]
\end{proof}

\begin{rem}
 The above proof does not work for $g=0$, because, in that case, $n=2$ and so  $C_\text{Cov}(p)=\infty$.
\end{rem}

The following proposition complements Proposition \ref{prop-top-comp-fin}.

\begin{prop} \label{prop-comp-sph}
The topological complexity of the sphere is finite.
\end{prop}

\begin{proof} Let $\mathbb{S}^2$ denote the $2$-sphere. Consider the abstract branch triplet $\mathcal{T}=(3,3,\{[3],[1,2],[1,2]\})$. By Lemma \ref{lem-ABT-realiz-sph-1}, $\mathcal{T}$ is realizable on a Riemann surface $M \in \mathbb{X}_{\mathbb{S}^2}$. Let $p \colon M \to \mathbb{P}^1$ be a branched covering whose branch data is equal to $\mathcal{T}$.

Since $n=3$, Corollary \ref{cor-hyp-BC-comp-fin} implies that $C_\text{Cov}(p)<\infty$. Hence, \[C_{\text{top}}(\mathbb{S}^2) \leq C_{\text{Riem}}(M) \leq C_\text{Cov}(p)<\infty.\]
\end{proof}
 
Combining Proposition \ref{prop-top-comp-fin} and Proposition \ref{prop-comp-sph}, we obtain the following.

\begin{thm} \label{thm-comp-all-surf-fin}
The topological complexity of every connected, closed, orientable surface is finite.
\end{thm}

\section{The main theorems} \label{sect-main-thms}
In this section we compute the topological complexity of all surfaces. We start with the case of surfaces with positive genus.


\begin{thm} \label{computing-complex}
Let $S$ be a connected, closed, orientable surface of genus $g \geq 1$. Then the topological complexity $C_{\text{top}}(S)$ of $S$ is equal to $2\pi(2g+1)$.
\end{thm}

\begin{proof} By Theorem \ref{thm-comp-all-surf-fin}, $C_{\text{top}}(S)<\infty$. Let $M$ be a Riemann surface in $\mathbb{X}_S$ such that $C_{\text{Riem}}(M)=C_{\text{top}}(S)$. Let $m_{\text{min}}$ is the minimum total length of an abstract branch triplet that is realizable on $M$. By Theorem \ref{main-Thm-CC-RR}, 

\begin{equation} \label{comp-complex-eq-1}
 C_{\text{Riem}}(M)=2\pi(m_{\text{min}}+2g-2)                                                                                                                                                                                                                                                                                                                                                                                                                                                                                                                                                                                                  \end{equation}

Let $d=2g+1$. Consider the abstract branch triplet $\mathcal{T}=(d,3,\{[d],[d],[d]\})$ of total length  $m_{\mathcal{T}}=1+1+1=3$. By Theorem \ref{thm-BT-realizable}, the abstract branch triplet $\mathcal{T}$ is realizable on a Riemann surface $M' \in \mathbb{X}_S$. Let $p' \colon M' \to \mathbb{P}^1$ be the $(d,n)$-branched covering associated to $\mathcal{T}$. Since $n=3$, Corollary \ref{cor-hyp-BC-comp-fin} implies that $C_\text{Cov}(p')<\infty$, and so $C_{\text{Riem}}(M') \leq C_\text{Cov}(p')<\infty$.

Let $m'_{\text{min}}$ be the minimum total length of an abstract branch triplet that is realizable on $M'$. By Theorem \ref{main-Thm-CC-RR},

\begin{equation} \label{comp-complex-eq-2}
 C_{\text{Riem}}(M')=2\pi(m'_{\text{min}}+2g-2).
\end{equation}

Since $C_{\text{Riem}}(M)=C_{\text{top}}(S) \leq C_{\text{Riem}}(M')$, then Equations \eqref{comp-complex-eq-1} and \eqref{comp-complex-eq-2} imply that $m_{\text{min}} \leq m'_{\text{min}}$. Further, $m'_{\text{min}} \leq m_{\mathcal{T}}=3$, as $\mathcal{T}$ is realizable on $M'$. Therefore, $m_{\text{min}} \leq 3$. On the other hand, by Proposition \ref{prop-mmin-bounded}, $m_{\text{min}} \geq 3$. Hence, $m_{\text{min}}=3$, and so we obtain: \[C_{\text{top}}(S)=C_{\text{Riem}}(M)=2\pi(3+2g-2)=2\pi(2g+1).\]
\end{proof}

To compute the topological complexity of the sphere we need the following lemma.

\begin{lem} \label{lem-top-comp-sph-1}
Let $p \colon M \to \mathbb{P}^1$ be a hyperbolic $(d,n)$-branched covering with $M$ homeomorphic to the sphere. Then:
\begin{enumerate}
 \item $d \geq 3$, and 
 \item $C_\text{Cov}(p) \geq 6 \pi$.
\end{enumerate}
\end{lem}

\begin{proof} 
$(1)$ By Corollary \ref{cor-hyp-BC-comp-fin}, it suffices to prove that if $d \leq 2$ then $n<3$. If $d=1$, then the branching set of $p$ is empty, and so $n=0$. Suppose that $d=2$. Then  $p \colon M \to \mathbb{P}^1$ is a $(2,n)$-branched covering and the genus of $M$ is $g=0$. It follows from Lemma \ref{lem-RH-hyp-surf} that $n=2g+2=2$.

$(2)$ By Corollary \ref{cor-hyp-BC-comp-fin}, $n \geq 3$. Further, $C_\text{Cov}(p)=2\pi d (n-2)$ (by Lemma \ref{lem-area-complex}). Finally, $d \geq 3$, by $(1)$, and thus, \[C_\text{Cov}(p)=2\pi d (n-2) \geq 2\pi (3) (3-2)=6 \pi.\]
\end{proof}

The following result complements Theorem \ref{computing-complex}.

\begin{thm} \label{computing-complex-genus-0}
The topological complexity of the sphere is equal to $6 \pi$.
\end{thm}

\begin{proof} Let $\mathbb{S}^2$ denote the $2$-sphere. Consider the abstract branch triplet $\mathcal{T}=(3,3,\{[3],[1,2],[1,2]\})$. By Lemma \ref{lem-ABT-realiz-sph-1}, $\mathcal{T}$ is realizable on a Riemann surface $M_0 \in \mathbb{X}_{\mathbb{S}^2}$. Let $p_0 \colon M_0 \to \mathbb{P}^1$ be a $(d,n)$-branched covering whose branch data is equal to $\mathcal{T}$.

Since $d=3$ and $n=3$, then Lemma \ref{lem-area-complex} implies that \[ C_{\text{Cov}}(p_0)=2\pi d (n-2)=2\pi(3) (3-2)=6 \pi.\] Hence,

\begin{equation} \label{eq-computing-comp-sph-1}
C_{\text{top}}(\mathbb{S}^2) \leq C_{\text{Riem}}(M_0) \leq C_{\text{Cov}}(p_0)=6 \pi.
\end{equation}

Let $M \in \mathbb{X}_{\mathbb{S}^2}$. We prove that $C_{\text{Riem}}(M) \geq 6 \pi$.  Suppose that $p \colon M \to \mathbb{P}^1$ is a $(d,n)$-branched covering. If $p$ is hyperbolic, then Lemma \ref{lem-top-comp-sph-1} implies that $C_\text{Cov}(p) \geq 6 \pi$. If $p$ is not hyperbolic, $C_{\text{Cov}}(p)=\infty$. In either case,  $C_{\text{Cov}}(p) \geq 6 \pi$. Thus, by definition, $C_{\text{Riem}}(M) \geq 6 \pi$.

Since $C_{\text{Riem}}(M) \geq 6 \pi$ for all $M \in \mathbb{X}_{\mathbb{S}^2}$, then, by definition, 

\begin{equation} \label{eq-computing-comp-sph-2}
C_{\text{top}}(\mathbb{S}^2) \geq 6 \pi.
\end{equation}

By equations \eqref{eq-computing-comp-sph-1} and \eqref{eq-computing-comp-sph-2}, we obtain that $C_{\text{top}}(\mathbb{S}^2)=6\pi$.
\end{proof}

We can merge Theorems \ref{computing-complex} and \ref{computing-complex-genus-0} into the following result. 

\begin{thm} \label{main-thm-top-comp}
Let $S$ be a connected, closed, orientable surface of genus $g$. Let $C_{\text{top}}(S)$ denote the topological complexity of $S$. Then 

\[C_{\text{top}}(S)=
\left\{
	\begin{array}{cl}
		2\pi(2g+1)   & \mbox{if } g \geq 1,\\
		 6 \pi  & \mbox{if } g=0.	
	\end{array}
\right.\]
\end{thm}

\section{The Simple Complexity of a Topological Surface} \label{sect-simp-comp}

In this section we define a special type of $(d,n)$-branched coverings that we call simple. As usual, $|B|$ denotes the cardinality of the set $B$.

\begin{defn} (\cite{CC-RR-BAMS}) \label{defn-simp-BC}
A $(d,n)$-branched covering $p \colon M \to \mathbb{P}^1$ is called \emph{simple} if $|p^{-1}(y)| \geq d-1$ for all $y \in \mathbb{P}^1$.
\end{defn}

For simple branched coverings, the Riemann-Hurwitz formula takes the following form.

\begin{thm} (\cite[Theorem $3.4$]{CC-RR-BAMS}) \label{thm-RH-simp-BC} Let $p \colon M \to S^2$ be a simple $(d,n)$-branched covering. Let $g$ be the genus of $M$. Then  \[2-2g=2d-n.\]
\end{thm}

We use Theorem \ref{thm-RH-simp-BC} to obtain a lower bound for the complexity of a simple branched covering.

\begin{lem} \label{lem-simp-top-comp-gen-pos}
Let $p \colon M \to \mathbb{P}^1$ be a simple $(d,n)$-branched covering. Suppose that the genus $g$ of $M$ is positive. Then $C_\text{Cov}(p) \geq 8 \pi g$.
\end{lem}

\begin{proof} If $n<3$, then  Corollary \ref{cor-hyp-BC-comp-fin} implies that $C_\text{Cov}(p)=\infty$ and so the conclusion of the lemma is clearly satisfied. Therefore, we can assume that $n \geq 3$ for the rest of the proof.

If $d=1$, then the branching set of $p$ is empty, which contradicts our assumption that $n \geq 3$. Therefore, $d \geq 2$.

Since $p$ is simple, Theorem \ref{thm-RH-simp-BC} implies that  $2-2g=2d-n$. This means that $n=2d+2g-2 \geq 4$, as $d \geq 2$ and $g \geq 1$. Therefore, by Lemma \ref{lem-area-complex},

\begin{equation} \label{eqn-lem-simp-top-comp-gen-pos-1}
 C_\text{Cov}(p)=2\pi d (n-2)=2\pi d (2d+2g-4)=4\pi d (d+g-2).
\end{equation}

Since $d \geq 2$ and $g \geq 1$, then equation \eqref{eqn-lem-simp-top-comp-gen-pos-1} implies that \[ C_\text{Cov}(p)=4\pi d (d+g-2) \geq 4\pi (2) (2+g-2)=8\pi g.\] 
\end{proof}

We now define a notion of simple complexity for topological surfaces.

\begin{defn} \label{defn-simp-comp-top}  Let $S$ be a connected, closed, orientable surface. We define the \emph{simple topological complexity}, $C_{\text{simp}}(S)$, of $S$ as the infimum of the set of all complexities of simple branched coverings of $M$ to $\mathbb{P}^1$ with $M \in \mathbb{X}_S$: \[ C_{\text{simp}}(S)=\inf\{C_{\text{Cov}}(p)\,|\, p \colon M \to \mathbb{P}^1 \text{ is a simple branched covering and } M \in \mathbb{X}_S \}\]
\end{defn}

Clearly, a double branched covering is simple. Thus, Proposition \ref{prop-XS-hyper} shows that the set $\{C_{\text{Cov}}(p)\,|\, p \colon M \to \mathbb{P}^1 \text{ is a simple branched covering and } M \in \mathbb{X}_S \}$ is not empty for every (connected, closed, orientable) surface $S$. Hence, Observation \ref{obs-non-empty-subset-C} implies that the infimum in Definition \ref{defn-simp-comp-top} is attained, although it could be infinite.

We now compute the simple topological complexity of all connected, closed, orientable surfaces. We start with the case of positive genus.

\begin{thm} \label{thm-comp-simp-complex-gen-greater-1}
Let $S$ be a connected, closed, orientable surface of genus $g \geq 1$. Then the simple topological complexity $C_{\text{simp}}(S)$ of $S$ is equal to $8\pi g$.
\end{thm}

\begin{proof} By Proposition \ref{prop-XS-hyper}, there exists a hyperelliptic Riemann surface $M_0 \in \mathbb{X}_S$. Let $p_0 \colon M_0 \to \mathbb{P}^1$ be the associated $(2,n)$-branched covering. By Lemma \ref{lem-RH-hyp-surf}, $n=2g+2$. Therefore, $g \geq 1$ implies that $n \geq 4$, and so \[C_\text{Cov}(p_0)=2\pi d (n-2)=2\pi (2) (2g)=8 \pi g.\]

Since $p_0$ is simple, then $C_{\text{simp}}(S) \leq C_\text{Cov}(p_0)=8 \pi g$. By Lemma \ref{lem-simp-top-comp-gen-pos}, $ C_{\text{simp}}(S) \geq 8 \pi g$. Combining the last two inequalities, we obtain that $ C_{\text{simp}}(S)=8 \pi g$.
\end{proof}

\begin{rem}
 The above proof does not work for $g=0$, because, in that case, $C_\text{Cov}(p_0)=\infty$.
\end{rem}

We now extend Theorem \ref{thm-comp-simp-complex-gen-greater-1} to the case $g=0$. To this end, we need the analog of Lemma \ref{lem-simp-top-comp-gen-pos} for genus zero.

\begin{lem} \label{lem-simp-top-comp-gen-0}
Let $p \colon M \to \mathbb{P}^1$ be a simple $(d,n)$-branched covering with $M \in \mathbb{X}_{\mathbb{S}^2}$. Then $C_\text{Cov}(p) \geq 12 \pi$.
\end{lem}

\begin{proof} We can assume that $n \geq 3$. Let $g$ be the genus of $M$. Since $p$ is simple and $g=0$, then Theorem \ref{thm-RH-simp-BC} implies that $n=2d-2$. In particular, $n$ is even. Combining this with the condition $n \geq 3$, we obtain that $n \geq 4$. This means that $d \geq 3$, and so \[C_\text{Cov}(p)=2\pi d (n-2) \geq 2\pi (3) (4-2)=12 \pi.\]
\end{proof}

We now state the analog of Theorem \ref{thm-comp-simp-complex-gen-greater-1} for genus zero.

\begin{thm} \label{thm-simp-top-comp-gen-0}
The simple topological complexity of the sphere is equal to $12\pi$.
\end{thm}

\begin{proof} Consider the abstract branch triplet $\mathcal{T}=(3,4,\{[1,2],[1,2],[1,2],[1,2]\})$. By Theorem \ref{thm-Bar}, there exists a $(3,4)$-branched covering $p_0 \colon M_0 \to \mathbb{P}^1$, with $M_0 \in \mathbb{X}_{\mathbb{S}^2}$, whose branch triplet coincides with $\mathcal{T}$. This branched covering is simple because the length of the partition $[1,2]$ is equal to $2=d-1$. Further, \[C_\text{Cov}(p)=2\pi d (n-2)=2\pi (3) (4-2)=12 \pi.\] Hence, $C_{\text{simp}}(\mathbb{S}^2) \leq C_\text{Cov}(p_0)=12 \pi$. By Lemma \ref{lem-simp-top-comp-gen-0}, $C_{\text{simp}}(\mathbb{S}^2) \geq 12 \pi$. Combining the last two inequalities, we obtain that $ C_{\text{simp}}(\mathbb{S}^2)=12\pi$.
\end{proof}

We can merge Theorems \ref{thm-comp-simp-complex-gen-greater-1} and \ref{thm-simp-top-comp-gen-0} into the following result. 

\begin{thm} \label{main-thm-simp-comp}
Let $S$ be a connected, closed, orientable surface of genus $g$. Let $C_{\text{simp}}(S)$ denote the simple topological complexity of $S$. Then 

\[C_{\text{simp}}(S)=
\left\{
	\begin{array}{cl}
		8\pi g   & \mbox{if } g \geq 1,\\
		 12\pi  & \mbox{if } g=0.	
	\end{array}
\right.\]
\end{thm}

\bibliographystyle{plain}
\bibliography{biblio}

\end{document}